\documentclass{amsart}
\usepackage{amsfonts, amsmath, latexsym, epsfig}
\usepackage[norelsize]{algorithm2e}
\usepackage{amssymb}
\usepackage{xcolor}
\usepackage{epsf}
\usepackage{url}
\usepackage{colonequals}
\usepackage{multirow}

\usepackage{hyperref}
\hypersetup{colorlinks=true,linkcolor=blue,anchorcolor=blue,citecolor=blue}

\newcommand{\RR}{\ensuremath{\mathbb{R}}}

\newcommand{\QQ}{\ensuremath{\mathbb{Q}}}
\newcommand{\CC}{\ensuremath{\mathbb{C}}}
\newcommand{\ZZ}{\ensuremath{\mathbb{Z}}}

\numberwithin{equation}{subsection}
\newtheorem{proposition}[equation]{Proposition}
\newtheorem{theorem}[equation]{Theorem}
\newtheorem{corollary}[equation]{Corollary}
\newtheorem{lemma}[equation]{Lemma}

\theoremstyle{definition}
\newtheorem{definition}[equation]{Definition}
\newtheorem{example}[equation]{Example}

\theoremstyle{remark}
\newtheorem{remark}[equation]{Remark}

\def\QuotS#1#2{\leavevmode\kern-.0em\raise.2ex\hbox{$#1$}\kern-.1em/\kern-.1em\lower.25ex\hbox{$#2$}}

\newcommand{\R}{\RR}
\newcommand{\Q}{\QQ}
\newcommand{\C}{\CC}
\newcommand{\Z}{\ZZ}

\definecolor{darkyellow}{HTML}{BDB76B}

\newcommand{\defi}[1]{\emph{\textsf{#1}}}

\newcommand{\calV}{{\mathcal V}}
\newcommand{\calS}{{\mathcal S}}
\newcommand{\calLmin}{{\mathcal L}_{\textup{min}}}

\newcommand{\tr}{\textup{\textsf{T}}{}\mkern-2mu}

\newenvironment{enumalph}
{\begin{enumerate}}
{\end{enumerate}}

\newenvironment{enumalg}
{\begin{enumerate}}
{\end{enumerate}}

\newenvironment{enumroman}
{\begin{enumerate}}
{\end{enumerate}}

\DeclareMathOperator{\maxdiag}{maxdiag}
\DeclareMathOperator{\Aut}{Aut}
\DeclareMathOperator{\satspan}{satspan}
\DeclareMathOperator{\Latt}{\satspan}
\DeclareMathOperator{\Span}{span}
\DeclareMathOperator{\M}{M}

\DeclareMathOperator{\Isom}{Isom}

\DeclareMathOperator{\GL}{GL}

\DeclareMathOperator{\Sp}{Sp}
\DeclareMathOperator{\disc}{disc}

\DeclareMathOperator{\Can}{Can}

\DeclareMathOperator{\proj}{proj}

\DeclareMathOperator{\Stab}{Stab}

\DeclareMathOperator{\Min}{Min}
\DeclareMathOperator{\CV}{CV}
\DeclareMathOperator{\cvd}{cvd}

\DeclareMathOperator{\Cls}{Cls}
\DeclareMathOperator{\Orth}{O}
\DeclareMathOperator{\Gen}{Gen}
\DeclareMathOperator{\Map}{Map}
\DeclareMathOperator{\SympBas}{SympBas}

\begin{document}

\author[Dutour Sikiri\'c]{Mathieu Dutour Sikiri\'c}
\address{Mathieu Dutour Sikiri\'c, Rudjer Boskovi\'c Institute, Bijenicka 54, 10000 Zagreb, Croatia}
\email{mathieu.dutour@gmail.com}

\author[Haensch]{Anna Haensch}
\address{Anna Haensch, Duquesne University, Department of Mathematics and Computer Science, Pittsburgh, Pennsylvania, USA}
\email{annahaensch@gmail.com}

\author[Voight]{John Voight}
\address{John Voight, Department of Mathematics, Dartmouth College, 6188 Kemeny Hall, Hanover, NH 03755, USA}
\email{jvoight@gmail.com}

\author[van Woerden]{Wessel P.J. van Woerden}
\address{Wessel van Woerden, Centrum Wiskunde \& Informatica, Science Park 123, 1098 XG Amsterdam, Netherlands}
\email{wvw@cwi.nl}

\title{A canonical form for positive definite matrices}

\begin{abstract}
We exhibit an explicit, deterministic algorithm for finding a canonical form for a
positive definite matrix under unimodular integral transformations.  We use characteristic sets of short vectors and partition-backtracking graph software.  The algorithm runs in a number of arithmetic operations that is exponential in the dimension $n$, but it is practical and more efficient than canonical forms based on Minkowski reduction. 
\end{abstract}

\maketitle

\section{Introduction}

\subsection{Motivation} \label{sec:motivation}

For $n$ a positive integer, let $\calS^n$ denote the $\R$-vector space of symmetric real $n\times n$-matrices and $\calS^n_{>0} \subset \calS^n$ denote the cone of positive definite symmetric $n\times n$-matrices.  For $A \in \calS^n_{>0}$, the map $x \mapsto x^{\tr} A x$ (where ${}^\tr$ denotes transpose) defines a positive definite quadratic form, with $A$ its Gram matrix in the standard basis; for brevity, we refer to $A \in \calS^n_{>0}$ as a \defi{form}.  The group $\GL_n(\ZZ)$ of unimodular matrices acts on $\calS^n_{>0}$ by the action $(U,A) \mapsto U^\tr AU$; the stabilizer of a form $A$ under this action is the finite group
\begin{equation}
\Stab(A) \colonequals \{ U\in \GL_n(\ZZ)  :  U^\tr A U = A\}.
\end{equation}
Two forms $A, B\in \calS^n_{>0}$ are said to be \defi{(arithmetically) equivalent} if there exists a unimodular matrix $U\in \GL_n(\ZZ)$ such that
\begin{equation}
A = U^\tr B U.
\end{equation}
In the Geometry of Numbers \cite{bookschurmann}, forms arise naturally as Gram matrices of Euclidean lattices under a choice of basis; in this context, two forms are arithmetically equivalent if and only if they correspond to isometric lattices.

Plesken--Souvignier \cite{PleskenSouvignier} exhibited algorithms to compute stabilizers and test for arithmetic equivalence among forms, and these have been used widely in practice \cite{Fluckiger_Cyclotomic,Kirschmer_OneClassGenera,GreenbergVoight,PerfectDim8,Schonenbeck_Hecke}.  In a more theoretical direction, Haviv--Regev \cite{HavivRegevLIP} proposed algorithms based on the Shortest Vector Problem and an isolation lemma for these purposes as well, with a time complexity of $n^{O(n)}$.

While these algorithms have been sufficient for many tasks, they suffer from an unfortunate deficiency. Suppose we have many forms $A_1,\dots,A_m \in \calS^n_{>0}$ and we wish to identify them up to equivalence. A naive application of an equivalence algorithm requires $O(m^2)$ equivalence tests (in the worst case).  The number of tests can be somewhat mitigated if useful invariants are available, which may or may not be the case.

Our approach in this article is to compute a \emph{canonical form} $\Can_{\GL_n(\ZZ)}(A)$ for $A \in \calS^n_{>0}$.  This canonical form should satisfy the following two basic requirements:
\begin{enumroman}
\item For every $A\in \calS^n_{>0}$, $\Can_{\GL_n(\ZZ)}(A)$ is equivalent to $A$; and
\item For every $A\in \calS^n_{>0}$ and $U \in \GL_{n}(\ZZ)$, $\Can_{\GL_n(\ZZ)}(U^\tr A U) = \Can_{\GL_n(\ZZ)}(A)$.
\end{enumroman}
(The equivalence in (i) is unique up to $\Stab(A)$.)  Combining a canonical form with a hash table, the identification of equivalence classes in a list of $m$ forms takes only $m$ canonical form computations (and $m$ hash table lookups) and so has the potential to be much faster.

\subsection{Minkowski reduction and characteristic sets}

The theory of Minkowski reduction provides one possible approach to obtain a canonical form.
The Minkowski reduction domain \cite{MinkowskiReduction} is a polyhedral domain $P_n \subset \calS^{n}_{>0}$ with the property that there exists an algorithm for \emph{Minkowski reduction}, taking as input a form $A$ and returning as output an equivalent form in $P_n$.  For example, for $n=2$ we recover the familiar Gaussian reduction of binary quadratic forms. An implementation of Minkowski reduction is available \cite{CaratSoftware}; however, this reduction is quite slow in practice, and it is unsuitable for forms of large dimension $n$ (say, $n \geq 12$).

For those forms whose Minkowski reduction lies in the \emph{interior} of the domain $P_n$, the Minkowski reduction is unique
\cite[p.\ 203]{DonaldsonMinkowski}, thereby providing a canonical form.
Otherwise, when the reduction lies on the boundary of $P_n$, there are finitely many possible Minkowski reduced forms; one can then order the facets of the polyhedral domain $P_n$ to choose a canonical form among them.
This approach was carried out explicitly by Seeber (in 1831) for $n=3$; and, citing an unpublished manuscript, 
Donaldson claimed ``Recently, Hans J.\ Zassenhaus has suggested that Minkowski reduction can be applied to the problem of row reduction
of matrices of integers''  \cite[p.\ 201]{DonaldsonMinkowski}.
An extension to $n=5,6,7$ is possible at least in principle, since $P_n$ is known in these cases \cite{bookschurmann}.  However, the problem of determining the
facets of the Minkowski reduction domain is hard in itself and so this strategy seems unrealistic in higher dimensions.
Other reduction theories \cite{zolotarev73,Grenier_Reduction_Theory} suffer from the same problem
of combinatorial explosion on the boundary.

In contrast, the approach taken by Plesken--Souvignier \cite{PleskenSouvignier} for computing the stabilizer and checking for equivalence of a form $A$ uses the following notion.

\begin{definition} \label{defn:charvectorset}
A \defi{characteristic vector set} function is a map that assigns to every $n \geq 1$ and form $A \in \calS_{>0}^n$ a finite subset of vectors $\calV(A) \subseteq \ZZ^n$ such that:
\begin{enumroman}
\item $\calV(A)$ generates $\ZZ^n$ (as a $\Z$-module); and
\item For all $U\in \GL_n(\ZZ)$, we have $U^{-1} \calV(A) = \calV(U^\tr A U)$.
\end{enumroman}
\end{definition}

The basic idea is then given a form $A$ to define an edge-weighted graph from a characteristic vector set $\calV(A)$;
using this graph, equivalence and automorphisms of forms becomes a problem about isomorphism and automorphisms of graphs (see Lemma \ref{lem:samestab}).
The graph isomorphism problem has recently been proved to be solvable in quasi-poly\-nomial time by Babai (see the exposition by Helfgott \cite{Helfgott_Babai}); however, the current approaches to computing characteristic vector sets (including ours) use algorithms to solve the Shortest Vector Problem which is known to be NP-hard \cite{Micciancio_SVP_NP_HARD}, so it is difficult to take advantage of this complexity result in the general case.  Nevertheless, we may hope to leverage some practical advantage from this approach.

\subsection{Our approach}

In this article, we adopt the approach of characteristic vector sets, using very efficient programs \cite{nauty,bliss} that compute a canonical form of a graph using partition backtrack.  A subfield $F$ of $\R$ is \defi{computable} if it comes equipped with a way of encoding elements in bits along with deterministic, polynomial-time algorithms to test equality, to perform field operations, and to compute (binary) expansions to arbitrary precision (for generalities, see e.g.\ Stoltenberg-Hansen--Tucker \cite{stoltenberg1999computable}).  For example, a number field with a designated real embedding is computable using standard algorithms.

\begin{theorem} \label{thm:mainthm}
There exists an explicit, deterministic algorithm that, on input a (positive definite) form $A \in \calS_{>0}^n$ with entries in a computable subfield $F \subset \R$, computes a canonical form for $A$. For fixed $n \geq 1$, this algorithm runs in a bounded number of arithmetic operations in $F$ and in a polynomial number of bit operations when $F=\Q$.
\end{theorem}

This theorem is proven by combining Proposition \ref{prop:matcanform} for the first
statement and Corollary \ref{cor:finitetime} for the running time analysis.  The running time in Theorem \ref{thm:mainthm} is exponential in $n$, as we rely on short vector computations; we are not aware of general complexity results, such as NP-hardness, for this problem.
In light of the comments about Minkowski reduction in the previous section, the real
content of Theorem \ref{thm:mainthm} is in the word \emph{explicit}.  We also find
this algorithm performs fairly well in practice (see section \ref{prac_complexity})---an implementation is available online \cite{polyhedral_cpp}.

\subsection{Contents}

In section \ref{SEC_Finite_set_invariant_vectors} we present the construction of some
characteristic vector set functions.  In section \ref{SEC_Finding_Canonical_Form},
we present how to construct a canonical form from a given characteristic set function.
In section \ref{SEC_complexity_full} we consider the time complexity of our algorithm; we conclude in section \ref{SEC_Applications} with extensions and applications.

\subsection{Acknowledgments}
This work was advanced during the conference \emph{Computational Challenges in the Theory of Lattices} at the Institute for Computational and Experimental Research in Mathematics (ICERM) 
and further advances were made during a visit to the Simons Institute for the Theory of Computing 
The authors would like to thank ICERM and Simons for their hospitality and support. Voight was supported by a Simons Collaboration grant (550029) and Van Woerden was supported by the ERC Advanced Grant 740972 (ALGSTRONGCRYPTO). We also thank Achill Sch\"urmann and Rainer Schulze-Pillot for help on Minkowski reduction theory and the anonymous referees for their detailed feedback.

\section{Construction of characteristic vector sets}\label{SEC_Finite_set_invariant_vectors}

In this section we build two characteristic vector set functions that can be used for the computation of
the stabilizer, canonical form, and equivalence of forms.

\subsection{Vector sets}

The sets of vectors that we use throughout this work are based on short or shortest vectors.
Given a set of vectors $\calV\subseteq \ZZ^n$, let $\Span(\calV)$ be the (not necessarily full) lattice
spanned over $\Z$ by $\calV$.
For $A\in \calS^n$ and $x\in \RR^n$, we write
\begin{equation}
A[x] \colonequals x^\tr A x\in \RR.
\end{equation}
For a form $A\in \calS^n_{>0}$ we define the \defi{minimum}
\begin{equation}
\min(A) \colonequals \min_{x\in \ZZ^n \smallsetminus \{0\}} A[x],
\end{equation}
the set of \defi{shortest} (or \defi{minimal}) \defi{vectors} and its span
\begin{equation} \label{eqn:calLMinsp}
\begin{aligned}
\Min(A) &\colonequals \left\{v\in \ZZ^n : A[v] = \min(A)\right\} \\
\calLmin(A) &\colonequals \Span(\Min(A)).
\end{aligned}
\end{equation}
The set of shortest vectors satisfies the desirable transformation property
\begin{equation}
\Min( U^\tr  AU) = U^{-1}\Min(A)
\end{equation}
for all $U \in \GL_n(\ZZ)$.  If $\Min(A)$ is full-dimensional, then $A$ is called \defi{well-rounded}.

Two obstacles remain for using $\Min(A)$ as a characteristic vector set:
\begin{enumerate}
\item[{\bf PB1.}] If $n\geq 2$, then $\Span(\Min(A))$ may not have rank $n$.
\item[{\bf PB2.}]  If $n\geq 5$, then $\Span(\Min(A))$ may have rank $n$ but may not equal $\ZZ^n$.
\end{enumerate}

Thus we have to consider other vector sets.  For $\lambda > 0$, let
\begin{equation}
\Min_A(\lambda) \colonequals \left\{v\in \ZZ^n \smallsetminus \{0\} : A[v] \leq \lambda \right\}.
\end{equation}

The vector set used for computing the stabilizer and automorphisms in the {\tt AUTO}/{\tt ISOM} programs of Plesken--Souvignier \cite{PleskenSouvignier} is:
\begin{equation}
\calV_{\textup{PS}}(A) \colonequals \Min_A(\maxdiag(A)),
\end{equation}
where $\maxdiag(A) \colonequals \max \{A_{ii} : 1\leq i\leq n\}$ is the maximum of the diagonal elements of $A$.  The vector set $\calV_{\textup{PS}}(A)$ contains the standard basis as a subset and as a result is adequate for computing the stabilizer. Typically LLL-reduction \cite{lenstra1982factoring} is used, leading to a decrease in $\maxdiag(A)$, to prevent large sets.
However, when computing equivalence we have a potential problem since two forms $A$ and $B$ can be equivalent but satisfy $\maxdiag(A) \neq \maxdiag(B)$. This is a limitation of {\tt ISOM}, which for equivalence can be resolved by taking the bound $\max\{\maxdiag(A), \maxdiag(B)\}$ (something we cannot do for our canonical form).

To prevent this problem we can use a more reliable vector set that consists of those vectors whose length is at most the minimal spanning length:
\begin{equation} \label{eqn:Vms}
\begin{aligned}
\calV_{\textup{ms}}(A) &\colonequals \Min_A(\lambda_{\min})\text{, where} \\
\lambda_{\min} &\colonequals \min \{\lambda>0 : \Span(\Min_A(\lambda)) = \ZZ^n\}.
\end{aligned}
\end{equation}
This vector set $\calV_{\textup{ms}}(A)$ is a characteristic vector set.
However, $\calV_{\textup{ms}}(A)$ can still be very large, making it impractical to use.
\begin{example} \label{exm:Alambda}
For example, the matrix
$A_{\lambda} = \left(\begin{array}{cc}
1 & 0\\
0 & \lambda
\end{array}\right)$ for $\lambda \geq 1$ gives
\begin{equation*}
\calV_{\textup{ms}}(A_{\lambda}) = \left\{\pm e_2\right\} \cup \{\pm e_1, \pm 2e_1, \dots, \pm \lfloor \sqrt{\lambda} \rfloor e_1\}.
\end{equation*}
while $\left\{\pm e_1, \pm e_2\right\}$ would be adequate. This problem is related to {\bf PB1}.
\end{example}

\subsection{An inductive characteristic vector set, using closest vectors} \label{sec:cvicvec}

Building on the observations made in the previous section, we now present a construction  that deals with {\bf PB1} and allows us to build a suitable characteristic vector set.

For a set of vectors $\calV \subseteq \ZZ^n$, the saturated sublattice (of $\Z^n$) spanned by $\calV$ is
\begin{equation}
\satspan(\calV) \colonequals \QQ \calV\cap \ZZ^n.
\end{equation}
Beyond shortest vectors, we use the \defi{closest vector distance}: for $v\in \QQ^n$, we define
\begin{equation}
\cvd(A, v) \colonequals \min_{x\in \ZZ^n} A[x - v]
\end{equation}
as the minimum distance from $\ZZ^n$ to the vector $v$ and
\begin{equation}
\CV(A,v) \colonequals \left\{x\in \ZZ^n : A[x-v] = \cvd(A,v) \right\}
\end{equation}
the set of closest vectors achieving this minimum.

Characteristic and closest vector sets behave well under restriction to a sublattice.  The following lemma describes this explicitly, in terms of bases.

\begin{lemma} \label{lem:indepentbasis}
Let $\calV$ be a characteristic vector set function, $A \in \calS_{>0}^n$ a form, and $L \subset \R^n$ a lattice of rank $r$.  Let $B \in \M_{n,r}(\R)$ be such that the columns are a $\Z$-basis of $L$; let $c$ be in the real span of $L$ and let $c_B \colonequals B^{-1} c \in \R^r$ be the unique vector such that $Bc_B = c$.  Then the sets
\[ B \calV( B^\tr  A B ) \quad \text{and} \quad B\CV(B^\tr  A B, c_B) \]
are independent of $B$ (depending only on $L,c$). 
\end{lemma}
\begin{proof}
The form $A|_B \colonequals B^\tr AB \in \calS_{>0}^r$ is the restriction of $A$ to $L$ in the basis $B$, so $B \calV(A|_B)$ is the characteristic vector set of this restricted form, as elements of $L \subset \R^n$.  Similarly, $B\CV(A|_B, c_B)$ is the set of vectors in $L \subset \R^n$, which are closest to $c$. Both sets only depend on $L$ and are independent of the chosen basis.
\end{proof}


Suppose that $A$ is well-rounded.  Let $v_1, \dots, v_n$ be a $\ZZ$-basis of the full rank lattice $\calLmin(A)$ spanned by $\Min(A)$ and let $B \in \M_{n \times n}(\Z)$ be the matrix with columns $v_1,\dots,v_n$.  We then define
\begin{equation}  \label{eqn:VcvA}
\calV_{\textup{wr-cv}}(A) \colonequals \Min(A) \cup \bigcup_{c\in \ZZ^n / \calLmin(A)} \left(c - B\CV(B^\tr  A B, B^{-1}c)\right).
\end{equation}
(It is possible to reduce the size of this set, e.g., by removing $0$ or filtering by length.)
The set $\calV_{\textup{wr-cv}}(A)$ consists of the union of the shortest vectors together with the set of points in each coset closest to the origin. By Lemma \ref{lem:indepentbasis}, the set $\calV_{\textup{wr-cv}}(A)$ is well-defined, independent of the choice of basis. Furthermore it satisfies the necessary transformation property and spans $\Z^n$ (as a $\Z$-module) because it contains at least one point from each coset in $\Z^n / \calLmin(A)$.

For a general form $A$, in geometrical terms we follow the filtration defined from the minimum \cite{CasselmanStability}. We define a set of vectors $\calV_{\textup{cv}}(A)$ inductively (described in an algorithmic fashion), as follows:
\begin{enumalg}
\item Compute the set $\Min(A)$ of vectors of minimal length and compute the saturated sublattice $L_1 \colonequals \Latt(\Min(A))$ spanned by these vectors.
\item Compute a $\ZZ$-basis $v_1, \dots, v_r$ of $L_1$, where $r$ is its rank.
  Let $B_1 \in \M_{n,r}(\R)$ be the matrix with columns $v_1,\dots,v_r$, and let $A_1 \colonequals B_1^{\tr} A B_1 \in \calS_{>0}^r$. Note that $A_1$ is well-rounded by construction.
\item 
Let $\proj \colon \ZZ^n \to \RR^n$ be the orthogonal projection on $L_1^{\perp}$ with respect to the scalar product defined by $A$.
\item Compute a basis $w_1, \dots, w_{n-r}$ of $L_2 := \proj(\Z^n)$ and let $B_2 \in \M_{n,(n-r)}(\R)$ the matrix with columns $w_1,\dots,w_{n-r}$. Let $A_2 \colonequals B_2^\tr  A B_2$.
\item If $r=n$, let ${\mathcal V}_{\textup{cv}}(A_2) \colonequals \emptyset$; otherwise, compute $\calV_{\textup{cv}}(A_2)$ recursively and let
\begin{equation} \label{Inductive_Definition_Vcv}
\calV_{\textup{cv}}(A) \colonequals B_1\calV_{\textup{wr-cv}}(A_1) \cup \bigcup_{v\in B_2\calV_{\textup{cv}}(A_2)} \CV(A, v).
\end{equation}
\end{enumalg}
\begin{theorem}\label{prop:vcv}
The following statements hold.
\begin{enumalph}
\item The set $\calV_{\textup{cv}}(A)$ is well-defined (independent of the choices of bases).
\item The association $A \mapsto \calV_{\textup{cv}}(A)$ is a characteristic vector set function.
\item We have $\#\calV_{\textup{cv}}(A) = n^{O(n)}$.
\item There is an explicit, deterministic algorithm that on input $A$ computes the set $\calV_{\textup{cv}}(A)$ in $n^{O(n)}$ arithmetic operations over $F$. For $F=\Q$ it has bit complexity $n^{O(n)} s^{O(1)}$ with $s$ the input size of $A$.
\end{enumalph}
\end{theorem}
\begin{proof}
We prove (a) by induction in the dimension $n$ that $\calV_{\textup{cv}}$ is a characteristic vector set. The base case $n=0$ is trivial.  For $n > 0$, note that $A_1$ is well rounded and $A_2$ has dimension at most $n-1$ and thus $B_1 \calV_{\textup{wr-cv}}(A_1)$ and $B_2 \calV_{\textup{cv}}(A_2)$ are independent of the choice of basis by induction and Lemma \ref{lem:indepentbasis}. The lattice $L_2$ is uniquely defined by the projection. 

For (b), by part (a), we may choose convenient bases. Running the algorithm for $A$ and $A' = U^\tr AU$ we can assume that $v_i' = U^{-1}v_i$ and $w_i' = U^{-1}w_i$ by using the transformation property of $\Min(A)$. Then $A_i' = A_i$ and $B_i' = U^{-1}B_i$ for $i=1,2$. We conclude by noting that $\CV$ also has the compatible transformation property:
\begin{equation}
	\CV(U^\tr AU, U^{-1}v) = U^{-1} \CV(A, v).
\end{equation}

For (c), By Keller--Martinet--Sch\"urmann \cite[Proposition 2.1]{MinkowskianSublattices} for a well-rounded lattice
the index of the sublattice determined by the shortest vectors is
at most $\lfloor \gamma_n^{n/2} \rfloor$ with $\gamma_n$ the Hermite constant satisfying $\gamma_n^{n/2} \leq (2/\pi)^{n/2} \cdot \Gamma(2+n/2) = n^{O(n)}$.
The bound on $\calV_{\textup{cv}}$ follows by combining this with exponential upper bounds
on the kissing number \cite{KabatjanskiLevenshtein} and the upper bound $2^n$ on $\#\CV(A,v)$ \cite[Proposition 13.2.8]{DL}.

The running time estimate (d) for arithmetic operations follows by combining single exponential upper estimates for algorithms to solve the CVP and SVP (see e.g.\ Micciancio--Voulgaris \cite{micciancio2013deterministic}).  We conclude with the bit complexity analysis for $F=\Q$. 
The bit complexity of SVP and CVP algorithms is indeed polynomial time in the input size \cite{helfrich1985algorithms,pujol2008rigorous}.  (We lack a reference for more general fields, and although we do not see major obstacles doing such an analysis, it would be out of the scope of this work).  For the computed projection, the Gram--Schmidt orthogonalization process also has a polynomial bit complexity in the input size (in bounded dimension, by induction).  The remaining steps in computing $\calV_{\textup{cv}}(A)$, including computing a basis out of a spanning set, computing a basis for the saturated sublattice, and computing representatives of the cosets $Z^n/\calLmin(A)$, are standard applications of the computation of a Hermite Normal Form (HNF)---see also section \ref{sec:canfm}.  A careful HNF computation can be achieved in polynomial time in the input size \cite{kannan1979polynomial}.  In particular, the obtained basis vectors and coset representatives also have a bit size that is polynomially bounded in the input size.  Thus for $F=\Q$ all arithmetic operations while computing $\calV_{\textup{cv}}(A)$ have a bit complexity polynomial in $s$.  We note for completeness that efficient versions of SVP, CVP, and HNF algorithms depend heavily on the famous LLL-algorithm.  
\end{proof}

Although the cost of computing many closest vector problems may make it quite expensive to compute $\calV_{\textup{cv}}(A)$ in the worst case, we find in many cases that it gives a substantial improvement in comparison to other characteristic vector sets.

\begin{example}
Returning to Example \ref{exm:Alambda}, we find that $\calV_{\textup{cv}}(A_{\lambda}) = \{ \pm e_1, \pm e_2\}$.
\end{example}



The construction of $\calV_{\textup{cv}}$ addresses {\bf PB1},
but {\bf PB2} remains---even for well-rounded lattices $\#(\Z^n / \calLmin(A))$ can possibly be very large.

\begin{example} \label{exm:Niemeier}
  The self-dual Niemeier lattice $N_{23}$ \cite[Chapter 18]{SPLAG}, whose root diagram is $24\mathsf{A}_1$  is well-rounded: it has minimum $2$ with $48$ shortest vectors, and $\#\calV_{\textup{ms}}(N_{23}) = 194352$.
  Since the index of the lattice spanned by the shortest vectors in $N_{23}$ is $2^{24}$, the size of $\calV_{\textup{cv}}(N_{23})$ is at least $48 + 2^{24}$.
\end{example}


\begin{remark}
It may be possible to deal with some cases (but still not Example \ref{exm:Niemeier}) by working with characteristic vector sets on forms attached in a canonical way to $A$: for example, one could work with the \emph{dual} form attached to $A$, for sometimes the dual has few minimal vectors (even if $A$ has many).
\end{remark}

\subsection{A characteristic vector set, using Voronoi-relevant vectors.} \label{sec:voricvec}
A well-known geometric shape associated to lattices is the \defi{Voronoi cell}. The Voronoi cell is the set of all points closer to $0$ with respect to $A$ than to any other integer point. For a form $A$, the (open) Voronoi cell is the intersection of half-spaces
\begin{equation}
	\text{Vor}(A) \colonequals \bigcap_{x \in \Z^n \setminus \{ 0 \}} H_{A,x},
\end{equation}
with $H_{A,x} \colonequals \{ y \in \R^n : A[y] < A[y-x] \}$. However, almost all vectors in this intersection are superfluous, and we only consider the set of \defi{Voronoi-relevant vectors} $\calV_{\textup{vor}}(A)$, i.e. the (unique) minimal set of vectors such that
\begin{equation}
	\text{Vor}(A) = \bigcap_{x \in \calV_{\textup{vor}}(A)} H_{A,x}.
\end{equation}
\begin{lemma} \label{lem:vorrel}
The following statements hold
\begin{enumalph}
	\item The association $A \mapsto \calV_{\textup{vor}}(A)$ is a characteristic vector set function.
	\item We have $\# \calV_{\textup{vor}}(A) \leq 2 \cdot (2^n-1)$.
	\item There is an explicit, deterministic algorithm that on input $A$ computes the set $\calV_{\textup{vor}}(A)$ in $2^{2n+o(n)}$ arithmetic operations over $F$. For $F = \Q$ it has bit complexity $2^{2n+o(n)}s^{O(1)}$ with $s$ the input size of $A$.
\end{enumalph}
\end{lemma}
\begin{proof}
Property \textup{(ii)} of a characteristic vector set for $\calV_{\textup{vor}}$ follows from the geometric definition, fully independent of the basis. For property \textup{(i)}, note that for any nonzero $x \in \Z^n$, we have $x \not\in \text{Vor}(A)$, and thus there is a vector $v \in \calV_{\textup{vor}}(A)$ such that $x-v$ lies strictly closer to $0$ with respect to $A$. Repeating this (a finite amount of time by a packing argument) we eventually end up at $0$ and thus $x$ is the sum of Voronoi-relevant vectors. The remaining statements follow from Micciancio--Voulgaris \cite{micciancio2013deterministic}.
\end{proof}
Although this characteristic vector set has great theoretical bounds, we refrain from using it in practice: most lattices actually attain the $2 \cdot (2^n-1)$ Voronoi bound,
whereas constructions based on short and close vectors often beat the theoretical worst-case bounds and give much smaller vector sets in practice.

\section{Construction of a canonical form}\label{SEC_Finding_Canonical_Form}

Suppose now that we have chosen a characteristic vector set function $\calV$, as in section \ref{sec:cvicvec} or \ref{sec:voricvec}.  From this, we will construct a canonical form, depending on $\calV$.

\subsection{Graph construction}

Given a form $A$, let $\calV(A)= \{v_1, \dots, v_p\}$.  We define $G_A$ to be the edge- and vertex-weighted complete (undirected) graph on $p$ vertices $1,\dots,p$ such that vertex $i$ has weight $w_{i,i}=A[v_i]$ and the edge between $i$ and $j$ has weight $w_{i,j}=v_i^\tr A v_j = w_{j,i}$.  In other words, $G_A$ is the weighted complete graph whose adjacency matrix is $B^{\tr} A B$, where $B \in \M_{n,p}(\R)$ is the matrix whose columns are $v_i$.  (The graph $G_A$ depends on $\calV$, but we do not include it in the notation as we consider $\calV$ fixed in this section.)

\begin{lemma} \label{lem:samestab}
For a form $A \in \calS_{>0}^n$ and the graph $G_{A}$ constructed from a characteristic vector set $\calV(A)$ we have a group isomorphism
\begin{equation}
	\Stab(A) \simeq \Stab(G_{A}) := \{ \sigma \in S_p : w_{i,j} = w_{\sigma(i), \sigma(j)} \text{ for all $1 \leq i,j \leq p$}\}.
\end{equation}
\end{lemma}
\begin{proof}
	We first define the map $\Stab(A) \to \Stab(G_A)$.  Let $U \in \Stab(A)$.  Then by property (ii) of a characteristic vector set, we have $U \calV(A) = \calV( U^{-\tr} A U^{-1} ) = \calV( A )$; therefore, $U$ permutes the set $\calV(A)$, giving a permutation $\sigma_U \in S_p$ characterized by $\sigma_U(i) = j$ if and only if $Uv_i = v_j$. Accordingly, we have
\begin{equation}
w_{i,j}=v_i^{\tr} A v_j = v_i^{\tr} U^{\tr} A U v_j = v_{\sigma_U(i)} A v_{\sigma_U(j)}
\end{equation}
so moreover $\sigma_U \in \Stab(G_A)$. It is then straightforward to see that this map defines a group homomorphism.
To show this map is an isomorphism, we use property (i) that $\calV(A)$ spans $\Z^n$.  Indeed, the map is injective because if $\sigma_U$ is the identity, then $U v_i = v_i$ for all $i$ so $U$ is the identity.  Similarly, it is surjective: any $\sigma \in \Stab(G_{A})$ fixes pairwise inner products with respect to $A$, so we obtain a unique $\Q$-stabilizer $U \in \GL_n(\Q)$ such that $U^{\tr} A U = A$; however, because $\calV(A)$ spans $\Z^n$, we obtain $U\Z^n=\Z^n$ so $U \in \Stab(A)$.
\end{proof}

\subsection{Graph transformations}
The software {\tt nauty} \cite{nauty} and {\tt bliss} \cite{bliss} allow to test equivalence and find the automorphism group and a canonical vertex ordering of vertex weighted graphs.
Thus, we need graph transformations that allow to translate our vertex and edge weighted complete graphs
into vertex weighted graphs (see also the {\tt nauty} manual \cite{nauty}).

Let $G$ be a complete (undirected) graph on $p$ vertices with vertex weights $w_{i,i}$ and edge weights $w_{i,j}$.
We construct a complete (undirected) graph $T_1(G)$ on $p+2$ vertices which is only edge weighted, as follows.
Let $a \colonequals 1 + \max_{i,j}w_{i,j}$ and $b \colonequals a + 1$ be two distinct weights that do not occur as $w_{i,j}$.
We define the new edge weight $w'_{i,j}$ for $i<j$ to be
\begin{equation}
  w'_{i,j} \colonequals
  \begin{cases}
   w_{i,j}, & \text{ if $i<j \leq p$;} \\
   w_{i,i}, & \text{ if $i \leq p$ and $j=p+1$;} \\
   a, & \text{ if $i \leq p$ and $j=p+2$;} \\
   b, & \text{ if $i=p+1$ and $j=p+2$}
   \end{cases}
\end{equation}
We have a natural bijection $\Isom(G,G') \xrightarrow{\sim} \Isom(T_1(G), T_1(G'))$ of morphisms in the categories of edge-and-vertex-weighted and edge-weighted graphs, hence taking $G'=G$, we have $\Aut(G) \simeq \Aut(T_1(G))$.

The next transformation takes a complete graph $G$ with edge weights $w_{i,j}$ and returns
a vertex weighted graph $T_2(G)$. Let $S$ be the list of possible edge weights, ordered
from the smallest to the largest, and let $w$ be
the smallest integer such that $\#S\leq 2^w$. For an edge weight $s\in S$, denote
$l_k(s)$ the $k$-th value in the binary expansion of the position of $s$ in $S$.
If $G$ has $p$ vertices then $T_2(G)$ will have $p w$ vertices of the form $(i,k)$
with $1\leq i\leq p$ and $0\leq k\leq w-1$.
The weight of the vertex $(i,k)$ is $k$.
Two vertices $(i,k)$ and $(i',k')$ are adjacent in the following cases:
\begin{enumerate}
\item $i=i'$, or
\item $k=k'$ and $l_k(w_{i,i'}) = 1$.
\end{enumerate}
Condition (i) implies that vertices of $G$ correspond to cliques in $T_2(G)$.
Condition (ii) means that each digit $k$ corresponds to a subgraph of $T_2(G)$.
We have again have a natural bijection $\Isom(G,G') \xrightarrow{\sim} \Isom(T_2(G), T_2(G'))$.

Combining this we can lift an isomorphism between $T_2(T_1(G_A))$ and $T_2(T_1(G_B))$ to an isomorphism
between $G_A$ and $G_B$ and thus to an isomorphism between $A$ and $B$ by solving an overdetermined
linear system.
Similarly, we can compute the group $\Aut(A)$ from $\Aut(T_2(T_1(G_A)))$.

\subsection{Canonical orderings of characteristic vector sets}
The canonical vertex ordering functionality of {\tt nauty} and {\tt bliss} gives an ordering
of the vertices of vertex weighted graphs. It is canonical in the sense that two isomorphic
graphs will after this reordering be identical.
We do not know a priori what this ordering is as it depends on the software, its version and
the chosen running options. We still call it \emph{canonical}, following standard terminology.

We need to lift the ordering of the vertex set of $T_2(T_1(G_A))$ into an ordering of the
vertex set of $G_A$ and so the characteristic vector set.
Every vertex $i$ of $G$ corresponds to a set $S_i$ of $w$ vertices in $T_2(G)$ with
$S_i\cap S_j=\emptyset$ for $i\not= j$. For two vertices $i,j$ of $G$ we set $i<j$
if and only if $\min S_i < \min S_j$ in the canonical vertex ordering of $T_2(G)$.
Similarly every vertex $i$ of $G$ maps to one vertex $\phi(i)$ of $T_1(G)$ with $\phi(i)\not=\phi(j)$
if $i\not= j$. Thus we set $i<j$ if and only if $\phi(i) < \phi(j)$ in the canonical ordering.

Combining the above we obtain a canonical ordering of the vertex set of $G_A$ and thus of the
characteristic vector set of the matrix $A$.

\subsection{Canonical form} \label{sec:canfm}


We have a canonical ordering of the characteristic vector set $\calV(A)$, which we write as $v_1, \dots, v_p$.
This ordering is only canonical up to $\Stab(A)$: for another canonical ordering, there is an element $S \in \Stab(A)$ such that $w_i=Sv_i$ for $i=1,\dots,p$, and conversely.  We will now derive a canonical form from the vectors $v_i$.

The Hermite Normal Form (HNF) of a matrix $Q \in \M_{m,n}(\Z)$ is the unique matrix $H=(h_{ij})_{i,j} \in \M_{m,n}(\Z)$
for which there exists $U \in \GL_m(\Z)$ such that $Q = U H$ and moreover:
\begin{enumroman}
\item The first $r$ rows of $H$ are nonzero and the remaining rows are zero.

\item For $1 \leq i \leq r$, if $h_{i,j_i}$ is the first nonzero entry in row $i$, then $j_1 < \ldots < j_r$.

\item $h_{i,j_i} > 0$ for $1\leq i\leq r$.

\item If $1 \leq k < i\leq r$, then $0\leq h_{k,j_i} < h_{i,j_i}$.

\end{enumroman}

In the cases that interest us, the matrix $Q_A$ with columns $v_1, \ldots, v_p$ defined by the characteristic vector set
$\calV(A)$ is of full rank and so the matrix $U$, obtained from the Hermite normal form $Q_A = UH$, is uniquely defined as well. Note that any other ordering $Sv_1, \ldots, Sv_p$ would lead to the matrix $SU$ for some $S \in \Stab(A)$. We denote the matrix $U$ by $U_{\calV(A)}$ and note that its coset representative in $\Stab(A) \backslash \!\GL_n(\ZZ)$ is well-defined (determined by $\calV(A)$).

We now define
\begin{equation} \label{eq:canform!}
\begin{aligned}
\Can_{\GL_n(\ZZ)}(A) &\colonequals U^{\tr}_{\calV(A)} A U_{\calV(A)} \in \calS_{>0}^n. 
\end{aligned}
\end{equation}
Then $\Can_{\GL_n(\Z)}(A)$ depends only on $\calV(A)$ and $A$. Proposition \ref{prop:matcanform} proves the first statement of our main result, Theorem \ref{thm:mainthm} (for any characteristic vector set function $\calV$).

\begin{proposition} \label{prop:matcanform}
The matrix $\Can_{\GL_n(\ZZ)}(A)$ is a canonical form for $A$.
\end{proposition}
\begin{proof}
Property (i) is clear by definition.  For (ii), given $P\in \GL_n(\ZZ)$, we have
\begin{equation}
U_{\calV(P^\tr A P)} \equiv U_{P^{-1} \calV(A)} \equiv P^{-1} U_{\calV(A)} \in \Stab(P^\tr AP) \backslash\!\GL_n(\ZZ).
\end{equation}
Thus $\Can_{\GL_n(\ZZ)}(P^\tr A P) = \Can_{\GL_n(\ZZ)}(A)$, as desired.
\end{proof}

\begin{remark}
An alternative to computing the canonical form would be to keep the canonicalized version of the graph $G_A$.
However, this graph can be quite large, and the positive definite form allows a more compact representation
even taking into account coefficient explosion that might occur with the Hermite normal form.
\end{remark}

\section{Analysis} \label{SEC_complexity_full}

\subsection{Theoretical time complexity}\label{SEC_complexity}

We now analyze the algorithmic complexity of computing a canonical form using the characteristic vector set in section \ref{sec:voricvec}.

\begin{theorem}
\label{thm:quasipoly}
Given as input a positive definite symmetric matrix $A \in \calS_{>0}^n$ with entries in a computable subfield $F \subset \R$, and a characteristic vector set $\calV(A)$, we can compute a canonical form for $A$ in time $\exp{( O(\log(N)^c)} + s^{O(1)}$ where $N \colonequals \#\calV(A)$, $s$ is the input size of $(A, \calV(A))$, and $c > 1$ is a constant.
\end{theorem}

\begin{proof}
Given the characteristic vector set $\calV(A)$ the corresponding graph can be computed in time polynomial in the input size of $A$ and $\calV(A)$ as this part is mostly dominated by the computation of $v^\tr Aw$ for $v,w \in \calV(A)$. Computing a Hermite normal form can be done in time polynomial in the matrix input size which is the same as $\calV(A)$ \cite{kannan1979polynomial}. Because the initial graph has at most $O(N^2)$ distinct weights the final constructed vertex-weighted graph $T_2(T_1(G_A))$ is of polynomial size in $N$. We can conclude if we have a quasi-polynomial algorithm to find a canonical form of a graph. For this we refer to a recent report by Babai \cite{BabaiCanonical}.
\end{proof}
{}
\begin{corollary} \label{cor:finitetime}
For all $n \geq 1$ and $A \in \calS_{>0}^n$ with entries in a computable subfield $F \subset \R$, we can compute a canonical form in at most $2^{O(n^c)}$ arithmetic operations in $F$ for some constant $c>1$. If $F=\Q$, the bit complexity is at most $2^{O(n^c)} + s^{O(1)}2^{O(n)}$ with $s$ the input size of $A$.
\end{corollary}
\begin{proof}
	By Lemma \ref{lem:vorrel} we have an characteristic vector set function $\calV_{\textup{vor}}$ such that $\calV_{\textup{vor}}(A)$ has cardinality at most $2(2^n-1)$ and can be computed in at most $2^{O(n)}$ arithmetic operations.  For the rational case, the bit complexity (and output size) is at most $s^{O(1)}2^{O(n)}$, with $s$ the input size of $A$. We conclude by  Theorem \ref{thm:quasipoly}.
\end{proof}


\subsection{Practical time complexity}\label{prac_complexity}

We give a short experimental review of the practial time complexity of our implementation \cite{polyhedral_cpp}. We selected a diverse set of test cases to benchmark our implementation: random forms, more than $500\,000$ perfect forms \cite{PerfectDim8} and more than $100$ special forms from the \emph{Catalogue of Lattices} \cite{catalogue}. For the random $n$-dimensional forms a basis matrix $B$ is constructed with entries uniform from $\{-n, \ldots, n\}$, which, if full rank, is turned into a form $A = B^TB$. The set of perfect forms contains all $10\,963$ perfect forms of dimension $2$ up to $8$ and in addition $524\,288$ perfect forms of dimension $9$. The set of special forms consists of a diverse subset from the Catalogue up to dimension $16$, including all laminated lattices. Up to dimension $20$ we used $32$-bit integers and above that (much slower) arbitrary precision integers to prevent overflow. The implementation currently supports the characteristic vector set function $\calV_{\textup{ms}}$ and has not been highly optimized. The main bottleneck seemed to be constructing the characteristic vector sets and the computation of all pairwise inner products (in arbitrary precision) for the graph. Perhaps surprisingly, determining the canonical graph itself took negligible time in most cases. In low dimensions where we can still use basic integer types, computing a canonical form takes a few milliseconds up to a few seconds. For random lattices we can expect relatively small characteristic sets even in large dimensions, therefore enumerating the minimal vectors quickly becomes the bottleneck in high dimensions. For special forms in higher dimensions such as the Leech lattice with $196\,560$ minimal vectors one can expect that the main bottleneck is related to the huge graph. Both storing the graph and computing a canonical representative might barely be in the feasible regime.

\begin{table}[]
\begin{tabular}{l|l|l||l|l|l||l|l|l}
\multicolumn{3}{c||}{} & \multicolumn{3}{c||}{Time (s)} & \multicolumn{3}{c}{$\#\calV_{\textup{ms}}$} \\ \hline
Type & Samples & $n$  & min & avg & max & min          & avg & max \\[1pt] \hline\hline
\multirow{2}{*}{Perfect} & 10\,963 & 2--8 &  0.00041   &  0.0032   &  0.086   &  6  &  73.74   &  240 \\
     &  524\,288  & 9  &  0.0039   &  0.00594   &  0.11  &  90  & 94.04 & 272 \\[1pt]
     \hline
\multirow{4}{*}{Random} & 100 & 10 &  0.0015 &  0.08   &  2.03  &  20  &  100.36  &  988   \\
 & 100 & 20 &   0.016  &  0.17   &  4.18   &  40  &  114.34   &  812   \\
 & 100 & 30 &   2.43  &  23.41   &  511.42  & 60  &  93.46   &  310 \\
 & 100 & 40 &   5.18  &  24.91   &  251.51   &  82    &  107.7   &  240   \\[1pt] \hline
Catalogue & 107 & 2-16 & 0.00018 & 2.12 & 36.71 & 4 & 630.47 & 4320 
\end{tabular}{}
\caption{Timings of our implementation \cite{polyhedral_cpp}.}
\end{table}

\section{Extensions and applications}\label{SEC_Applications}

We conclude with an extension and a description of some applications.

\subsection{Extension to symplectic groups}\label{SEC_Extensions}


Let $J_n \colonequals \begin{pmatrix}
  0    & I_n\\
  -I_n & 0
  \end{pmatrix}$ represent the standard alternating pairing and
\begin{equation}
\Sp_{2n}(\ZZ) \colonequals \left\{Q\in \GL_{2n}(\ZZ) :  Q^\tr J_n Q = J_n\right\}.
\end{equation}
The group $\Sp_{2n}(\ZZ)$ acts on $\calS^{2n}_{>0}$ and we seek a canonical form for this action \cite{ExplicitReductionSP4}.


\begin{theorem}\label{SymplecticBasisAlgorithm}
Given a ordered set of vectors ${\mathcal V} = (v_1, \dots, v_m)$ that generates $\ZZ^{2n}$ as a lattice, there exists an effectively computable symplectic basis $\SympBas({\mathcal V})$ of $\ZZ^{2n}$ such that for every $P\in \Sp_{2n}(\ZZ)$ we have $\SympBas({\mathcal V}P) = \SympBas({\mathcal V}) P$.
\end{theorem}
\begin{proof}
Let $w_1$ be the first non-zero vector in ${\mathcal V}$ divided by the gcd of its coefficients.  
Since the family of vectors spans $\ZZ^n$, the gcd of the symplectic products $\omega(w_1, v_j)$ is $1$.
Thus we can find in a deterministic manner integers $\alpha_i$ such that $w_{2n} = \sum_{i=1}^m \alpha_i v_i$
satisfies $\omega(w_1, w_{2n}) = 1$.
We can then replace the vectors $v_i$ of the vector family by $v'_i = v_i - \omega(v_i, w_{2n}) w_1 + \omega(v_i, w_1) w_{2n}$. They satisfy $\omega(v'_i, w_1) = \omega(v'_i, w_{2n})=0$.
Thus we apply the same construction inductively on them and get our basis. The invariance property follows from the fact that we never use specific coordinate systems.
\end{proof}

A canonical representative for a form $A \in \calS^{2n}_{>0}$ under the action of $\Sp_{2n}(\Z)$ can also be computed using our canonical form, as follows:
\begin{enumalg}
\item Compute a characteristic vector family using e.g. $\calV_{\textup{cv}}$.
\item Compute a graph on this characteristic set of vector by assigning to two vectors $v$, $v'$ the weight $(vAv', vJ_n v')$.
\item Apply the canonicalization procedure and get a canonical ordering of $\calV_{\textup{cv}}$.
\item Use Theorem \ref{SymplecticBasisAlgorithm} in order to get a symplectic basis which then gives a reduction matrix.
\end{enumalg}

\subsection{Lattice databases}

Several efforts have sought to enumerate lattice genera of either bounded discriminant or satisfying some arithmetic conditions such as small (spinor) class number.  For example, the Brandt--Intrau tables \cite{G63} of reduced ternary forms with discriminant up to 1000, Nipp's tables \cite{N91} of positive definite primitive quaternary quadratic forms with discriminant up to 1732, and more recently the complete table of lattices with class number one due to Kirschmer--Lorch \cite{KL16}, to name a few.  A current project of interest in number theory is an extension of the L-functions and Modular Forms DataBase (LMFDB) \cite{lmfdb} to include lattices.  

The general strategy for generating these tables can take several forms.  For example, a list of isometry class candidates can be generated by extending lattices of lower rank in some systematic way \cite{G63,N91}.
Classes can also be generated by Kneser's method of neighboring lattices \cite{SP91} (see section \ref{sec:algmodfrm} below).  Although the completeness of the list of genus representatives can be verified using the Minkowski--Siegel mass formula, one critical bottleneck in most of these schemes is eliminating redundancy in the lists generated, especially for lattices with high rank and class number---it is here where we profit significantly from a canonical form.

Another current shortcoming of the database has been the lack of a deterministic naming scheme for lattices.  Although lattices up to equivalence can be classified by dimension, determinant, level, and class number,  beyond that point many genera of such lattices can exist, and each genus can potentially contain multiple classes.  Finding a canonical form for lattices provides a way to establish a deterministic labeling.  This has long been known to be a challenge: for example, it is exactly the problem of the boundary of a fundamental domain in Minkowski reduction (mentioned in the introduction) that is at issue.  
Ad hoc enumeration and labeling suffers from the deficiency that a computer failure or other issues in the database could result in new and different enumeration.  A canonical form provides a mechanism for a canonical label for lattices.  
Such a scheme would still depend on the graph canonical form being called in the algorithm; but in the event of a switch 
a bijective dictionary could easily be stored between the new naming and the old, giving still a nearly permanent 
deterministic naming of lattices. 

\subsection{Application to enumeration of perfect forms}
A canonical form really shows its strength compared to pairwise equivalence checks when the number of forms to be classified becomes very large. This is certainly the case during the enumeration of perfect forms using Voronoi's algorithm in dimension $9$ or higher. In dimension $9$ already more than $20$ million (inequivalent) perfect forms are found and the total number could be on the order of half a billion \cite{vanwoerdenperfect}. Even though there are some useful invariants such as the number of miminal vectors, the determinant and the size of the automorphism group, the number of remaining candidates for equivalence for each found perfect form can become quite large.  Removing equivalent forms is a large part of the computational cost during the enumeration. 

Therefore, efficiently finding a canonical form seems to be a necessity in completing the full enumeration in dimensions $9$ or higher. Luckily by the definition of a perfect form we always have that $\Min(A)$ is full dimensional. Furthermore for all perfect forms found so far $\Min(A)$ also spans $\ZZ^n$ and therefore the function $\calV_{\textup{ms}}$ seems to be an efficient way to obtain a small characteristic vector set. In Section \ref{prac_complexity} we saw that computing a canonical perfect form in dimension $9$ takes just a few milliseconds.


\subsection{Application to algebraic modular forms} \label{sec:algmodfrm}

Finally, we present an application to speed up computations of orthogonal modular forms, a special case of the theory of \emph{algebraic modular forms} as defined by Gross \cite{gross1999algebraic}.  We shift our perspective slightly, 
varying lattices in a (fixed) quadratic space.

Let $L \subset V$ be a \defi{(full) lattice}, the $\Z$-span of a $\Q$-basis for $V$.  
We say $L$ is \defi{integral} if $x^\tr  Ay \in \Z$ for all $x,y \in L$, and suppose that $L$ is integral.  We represent $L$ in bits by a basis $\{v_1,\dots,v_n\}$; letting $U_L$ be the change of basis matrix, we obtain a  form
\begin{equation} 
A_L \colonequals (v_i^\tr  A v_j)_{1 \leq i,j \leq n} = U_L^\tr  A U_L.
\end{equation}
(It is not necessarily the case that $A_L$ is arithmetically equivalent to $A$---the change of basis need only belong to $\GL_n(\Q)$.)  

In order to organize these lattices, we define the \defi{orthogonal group}
\begin{equation}
\Orth(V) \colonequals \{P \in \GL_n(\Q) : P^\tr  A P = A\}.
\end{equation}
Integral lattices $L,L' \subset V$ are \defi{isometric}, written $L \simeq L'$, if there exists $P \in \Orth(V)$ such that $P(L)=L'$.  Choosing bases for $L,L'$, we see that $L \simeq L'$ if and only if $A_L$ and $A_{L'}$ are arithmetically equivalent.

We repeat these definitions replacing $\Q$ (and $\Z$) by $\Q_p$ (and $\Z_p$) for a prime $p$, abbreviating $L_p \colonequals L \otimes_\Z \Z_p$.  Then the \defi{genus} of $L$ is
\begin{equation}
\Gen(L) \colonequals \{ L' \subset V : L_p \simeq L'_p \text{ for all primes $p$}\}.
\end{equation}
Finally, we define the \defi{class set} $\Cls(L)$ as the set of isometry classes in $\Gen(L)$.  By the geometry of numbers, we have $\#\Cls(L)<\infty$.

The theory of $p$-neighbors, due originally to Kneser \cite{kneser1957klassenzahlen}, gives an effective method to compute representatives of the class set $\Cls(L)$, as follows.  Let $p$ be prime (allowing $p=2$) not dividing $\det(A_L)$.  We say that a lattice $L' < V$ is a \defi{$p$-neighbor} of $L$,  and write $L' \sim_p L$, if $L'$ is integral  and
\begin{equation}
[L:L \cap L']=[L':L \cap L']=p
\end{equation}
(index as abelian groups).  If $L \sim_p L'$, then $\disc(L)=\disc(L')$ and $L' \in \Gen(L)$ \cite[Lemma 5.7]{GreenbergVoight}.  
The set of $p$-neighbors can be computed in time $O(p^{m+\epsilon} H_n(s))$, where $s$ is the input size and $H_n$ is a polynomial depending on $n$. 
Moreover, by strong approximation \cite[Theorem 5.8]{GreenbergVoight}, there is an effectively computable finite set $S$ of primes such that every $[L'] \in \Cls(L)$ is an \defi{iterated $S$-neighbor}
$L \sim_{p_1} \cdots \sim_{p_r} L_{r} \simeq L'$ with $p_i \in S$.  Typically, 
we may take $S=\{p\}$ for any $p \nmid \disc(L)$.  In this way, we may compute a set of representatives for $\Cls(L)$ from iterated $S$-neighbors.

The space of \defi{orthogonal modular forms for $L$} (with trivial weight) is
\begin{equation} 
M(\Orth(L))  \colonequals \Map(\Cls(L),\C). 
\end{equation}
In the basis of characteristic functions $\delta_{[L']}$ for $[L'] \in \Cls(L)$ we have $M(\Orth(L)) \simeq \C^h$ where $h \colonequals \#\Cls(L)$.  For $p \nmid \disc(L)$, define the \defi{Hecke operator}
\begin{equation}
\begin{aligned}
T_p \colon M(\Orth(L)) &\to M(\Orth(L)) \\
T_p(f)([L']) &= \sum_{M' \sim_p\, L'} f([M']).
\end{aligned}
\end{equation}
The operators $T_p$ commute and are self-adjoint (with respect to a natural inner product);
accordingly, there exists a basis of simultaneous eigenvectors for the Hecke operators, called \defi{eigenforms}.  

In this way, to compute the matrix representing the Hecke operator $T_p$, for each $[L'] \in \Cls(L)$, we need to identify the isometry classes of the $p$-neighbors of $L'$.  Here is where our canonical form algorithm applies, returning to our original motivation: after computing canonical forms for $\Cls(L)$, for each $p$-neighbor, we compute their canonical forms and then a hash table look up on $\Cls(L)$. This reduces our computation from $O(h^2)$ \emph{isometry tests} to $O(h)$ \emph{hash table lookups}.  For medium-sized values of $n$, we hope that the use of canonical forms will allow us to peer more deeply into the world of automorphic forms on orthogonal groups.






\end{document}